\documentclass[12pt]{article}
\usepackage{geometry} 
\usepackage[ utf8 ]{inputenc}
\usepackage[T1]{fontenc}
\usepackage[english,french]{babel}
\usepackage{amsmath}
\usepackage{amsfonts}
\usepackage{amssymb}
\usepackage{amsthm}
\usepackage{textcomp}
\usepackage{eucal}
\usepackage{array}
\theoremstyle{plain}
\newtheorem{theorem}{Théorème}
\newtheorem{corollary}{Corollaire}

\newtheorem{proposition}{Proposition}
\newtheorem{notation}{Notation}
\theoremstyle{definition}
\newtheorem{definition}{Définition}
\newtheorem{hypothesis}{Hypothèse}
\theoremstyle{remark}
\newtheorem{remark}{Remarque}
\date{}

\title{ Représentations de réflexion de groupes de Coxeter\\Première partie: le cas irréductible}
\author{François ZARA
}
\begin{document}
\maketitle
\begin{abstract}
Dans ce travail on étudie des représentations de certains groupes de Coxeter pour en déduire des propriétés des groupes de réflexions correspondants.

\end{abstract}
\begin{otherlanguage}{english}
\begin{abstract}
In this work we study representations of certain Coxeter groups to obtain some properties of the corresponding reflection groups.
\end{abstract}
\end{otherlanguage}
\footnote{Mots clés et phrases: groupes de Coxeter, groupes de réflexion, construction  de représentations de réflexion de groupes de Coxeter.}
\section{Introduction}
Dans ce travail on étudie certaines représentations de groupes de Coxeter de rang fini, irréductibles et $2$-sphériques: tous les $m_{st} $ (ordre du produit des réflexions $s$ et $t$) sont finis (conditions H(Cox)). Dans cette généralité, le problème est hors de portée, aussi nous nous restreignons à certaines classes de représentations. Soient $K$ un corps de caractéristique $0$, (W,S) un système de Coxeter qui satisfait aux conditions précédentes et $M$ un $K$-espace vectoriel. Soit $R$: $W \to GL(M)$ une représentation de $W$ qui satisfait aux conditions H(R) suivantes:
\begin{enumerate}
  \item  $\forall s \in S$, $R(s)$ est une réflexion de $M$;
  \item Si $s\in S$ et si $a_{s}$ est un vecteur directeur de $R(s)$, alors $\emph{A}:=\{a_{s}|s\in S\}$ est une base de $M$;
  \item  $\forall (s,t)\in S\times S$, $R(s)R(t)$ a le même ordre que $st$.
\end{enumerate}
dans ces conditions, on dit que $R$ est une \textbf{ représentation de réflexion } de $W$. On remarque que $R(W)$ est un sous-groupe de $GL(M)$ engendré par des réflexions.\\
Le premier but de ce travail est de construire \textbf{toutes} (à équivalence près) les représentations de réflexion de $W$. Pour cela on utilise les outils et résultats suivants:

	1) Soient $r$ et $s$ deux réflexions de $M$, de vecteurs directeurs $a$ et $b$ respectivement. Il existe $\lambda$ et $\mu$ dans $K$ tels que:\\
	$r(a)=-a, \quad r(b)=b+\lambda a$ \qquad et \qquad $s(a)=a+\mu b, \quad s(b)=-b$.\\ Alors $rs$ est d'ordre fini $n\geqslant 3$ si et seulement si il existe un entier $k$ premier à $n$ tel que $\lambda\mu=4\cos ^{2}(\frac{k\pi}{n})$ et $rs$ est d'ordre $2$ si et seulement si $\lambda = \mu =0$. L'élément $4\cos ^{2}(\frac{k\pi}{n})$ est racine d'un polynôme unitaire à coefficients entiers $u_{n}(X)$ et cette famille de polynômes est une famille de polynômes orthogonaux. On a un facteur irréductible de $u_{n}(X)$: $v_{n}(X)$ dont les racines sont celles pour lesquelles $(k,n)=1$. Ces polynômes ont été définis d'abord dans \cite{Z}.

2) A chaque groupe de Coxeter $W$, on peut associer un graphe $\Gamma(W)$ dont les sommets sont les éléments de $S$ et $(s,t)$ est une arête  de $\Gamma(W)$ si l'ordre $m_{st}$ de $st$ est $\geqslant 3$. On décore chaque arête du symbole $m_{st}$ . Si $W$ est irréductible, $\Gamma(W)$ est connexe. On appelle aussi $\Gamma(W)$ un diagramme.\\
Pour construire la représentation $R$ satisfaisant aux conditions H(R), on choisit un arbre couvrant $\textit{T}$ et $s_{0}$ un sommet (une racine) de $\Gamma(W)$ . On appelle $E(\Gamma(W))$ ou $E(\Gamma)$, l'ensemble des arêtes de $\Gamma(W)$ et $E(\textit{T})$ l'ensemble des arêtes de $\textit{T}$, enfin on pose $E'(\textit{T}) :=E(\Gamma(W))-E(\textit{T})$.\\
On définit sur l'ensemble des sommets de $\Gamma(W)$ une relation d'ordre, notée $\preccurlyeq$ (qui dépend de $\textit{T}$ et de $s_{0}$) de la manière suivante: $s\preccurlyeq t$ si $s$ et $t$ sont sur la même branche de l'arbre $\textit{T}$ et si la distance de $s$ à $s_{0}$ (dans $\textit{T}$) est $\leqslant $ à la distance de $t$ à $s_{0}$ (dans $\textit{T}$).\\
On construit maintenant la représentation $R$. Pour cela on fait un certain nombre de choix.
\begin{enumerate}
  \item $\forall e \in E(\Gamma(W))$, $\alpha_{e}$ une racine de $v_{m_{e}}(X)$;
  \item soit $e:=(s,t)\in E'(\textit{T})$. On pose $e':=(t,s)$ (notation seulement) et on choisit $l_{e}$ et $l_{e'}$ dans $K$ de telle sorte que $l_{e}l_{e'}=\alpha_{m_{e}}$.

\end{enumerate}
On appelle $K_{0}$ un corps de décomposition de l'ensemble des polynômes $u_{m_{e}}(X)$, $e$ parcourant l'ensemble des arêtes de $E(\Gamma(W))$. On peut choisir $K_{0}$ comme sous-corps de $\mathbb{R}$, et alors $K_{0}$ est un sous-corps réel d'un corps cyclotomique. On choisit  alors $K$ comme engendré par $K_{0}$ et tous les $l_{e}$, $e \in E'(\textit{T})$.

On définit des éléments $\zeta_{s}$ ($s\in S$) de $GL(M)$ de la manière suivante:
\begin{itemize}
  \item $\forall s \in S, \zeta_{s}(a_{s})=-a_{s}$;
  \item soit $(s,t)\in S^{^2}$ tel que $st=ts\neq 1$, on pose $\zeta_{t}(a_{s})=a_{s}$;
  \item soit $e:=(s,t)\in E(\textit{T})$ avec $s\preccurlyeq t$, on pose $\zeta_{s}(a_{t})=\alpha_{e}a_{s}+a_{t}$ et $\zeta_{t}(a_{s})=a_{s}+a_{t}$
  \item Soit $e:=(s,t)\in E'(\textit{T})$, on pose $\zeta_{s}(a_{t})=l_{e}a_{s} +a_{t}$ et $\zeta_{t}(a_{s})=l_{e'}a_{t} +a_{s}$
\end{itemize}
\begin{theorem}
L'application $R :s \mapsto \zeta_{s}: S \to \{\zeta_{s} | s \in S\}$ se prolonge en une représentation (notée aussi $R$) $R: W \to GL(M)$ qui possède par construction les propriétés H(R). On dit que $R$ est obtenue par la \textbf{construction fondamentale}. On pose $G:= Im(R)$.\\
Si l'on change d'arbre couvrant ou bien si l'on change de racines, on obtient une représentation équivalente. Par contre si l'on change l'un des $\alpha_{e}$ ou bien l'un des $l_{e}$ on obtient des représentations inéquivalentes. De plus on obtient ainsi toutes les représentations de réflexion de $W$ qui satisfont à H(R) à équivalence près.
\end{theorem}
Soient $\sigma$ un automorphisme de $K$ et  $\Phi$ l'espace vectoriel des formes $\sigma$-sesquilinéaires invariantes par $G$. Alors $dim\, \Phi \leqslant 1$. Dans la plupart des cas on a $dim\, \Phi =0$. C'est l'une des raisons pour lesquelles on n'utilisera pas les éléments de $\Phi$ lorsque celui-ci est $\neq 0$.

Avec les hypothèses du théorème, le changement de racine permet de distinguer les systèmes de racines de type $B_{n}$ et $C_{n}$.
\section{Généralités sur les groupes de réflexion}
\subsection{Quelques propriétés du produit de deux réflexions}
On rappelle ici qu'une réflexion d'un espace vectoriel $M$ sur un corps $K$ de caractéristique $\neq 2$ est un élément $r$ de $GL(M)$, d'ordre $2$, tel que $H(r):=Ker(r-Id_{M})$ est un hyperplan de $M$. Un générateur de $Im(r-Id_{M})$ s'appelle un vecteur directeur de $r$. On a $M=H(r)\oplus Im(r-Id_{M})$. On pose: $<v_{r}^{-}>=Im(r-id_{M})$.

Dans toute la suite on utilise les résultats et notations de \cite{Z}. 

On fait l'hypothèse H(1) suivante, valable dans toute la suite de ce travail:
\begin{hypothesis}\label{H1}
$K$ est un corps de caractéristique $0$ et $M$ est un $K$-espace vectoriel de dimension finie.

\end{hypothesis}
Soient $r$ et $s$ deux réflexions de $M$, de vecteurs directeurs $a$ et $b$ respectivement et d'hyperplans de points fixes $H(r)$ et $H(s)$. Si $(a,b)$ est un système libre, on peut écrire:\\
$r(a)=-a, \quad r(b)=b+c(r,a;s,b) a$ \qquad et \qquad $s(a)=a+c(s,b;r,a) b, \quad s(b)=-b$, avec $c(r,a;s,b)$ et $c(s,b;r,a)$ deux éléments de $K$.
\begin{definition}\label{D1}
On pose $C(r,s):=c(r,a;s,b)c(s,b;r,a)$ et on appelle $C(r,s)$ le \textbf{coefficient de Cartan} du couple $(r,s)$.
\end{definition}
\begin{proposition}\label{P1}
On garde les hypothèses et notations précédentes. On a:
\begin{enumerate}
  \item $C(r,s)$ ne dépend que de $r$ et $s$ et $C(r,s)=C(s,r)$.
  \item \begin{enumerate}
  \item $rs$ est d'ordre fini $\geqslant 3$ si et seulement si il existe un entier $k$ premier à $n$ tel que $C(r,s)=4\cos ^{^2}\frac{k\pi}{n}$ si et seulement si $C(r,s)$ est racine du polynôme $v_{n}(X)$. 
  \item $rs$ est d'ordre $2$ si et seulement si $c(r,a;s,b)=c(s,b;r,a)=0$.
  
\end{enumerate}

\end{enumerate}
\end{proposition}
\begin{proof}
Pour le 1), on peut remarquer que si l'on remplace $a$ par $\lambda a$ et $b$ par $\mu b$ avec $\lambda\mu \in K^{*} (= K-\{0\}) $ alors on a $c(r,\lambda a;s,\mu b)=\lambda^{-1}\mu c(r,a;s,b)$ et $c(s,\mu b;r,\lambda a)=\lambda\mu^{-1}c(s,b;r,a)$. \\
Il est clair que $C(r,s)$ = $C(s,r)$.\\
Le 2) est bien connu. On peut en trouver une démonstration dans \cite{Z}.
\end{proof}
\begin{notation}
Soient $K$ un corps, $M$ un $K$-espace vectoriel de dimension finie et $g$ un élément de $GL(M)$. On appelle $P_{g}(X)$ le polynôme caractéristique de $g$.
\end{notation}
\begin{proposition}\label{P2}
On garde les hypothèses (H1) et on suppose que $M$ est de dimension $m$. Soient $r$ et $s$ deux réflexions de $M$ de vecteurs directeurs $a$ et $b$ respectivement. On suppose que $(a,b)$ est un système libre. Alors:
\[
P_{rs}(X)=(X-1)^{m-2}(X^{2}-(-2+C(r,s))X+1)
\]
En particulier la trace de $rs$ est $Tr(rs)=m-4+C(r,s)$.
\end{proposition}
\begin{proof}
Il existe des éléments $c_{i}$ de $M$ $(1\leqslant i\leqslant m-2)$ tels que\\ $(a,b,c_{1}\cdots c_{m-2})$ soit une base de $M$ et $r(c_{i})=s(c_{i})=c_{i}$ $\forall i$. Il suffit donc de se placer dans le sous-espace de $M$ engendré par $a$ et $b$.
On a alors:
\[
r=\left(\begin{array}{cc}-1 & c(r,a;b,s) \\0 & 1\end{array}\right), s=\left(\begin{array}{cc}1 & 0 \\c(s,b;r,a) & -1\end{array}\right)
 \]
 d'où 
  \[ 
  rs=\left(\begin{array}{cc}-1+C(r,s) & -c(r,a;s,b) \\c(s,b;r,a) & -1\end{array}\right)
  \]
donc $P_{rs}(X)=(X-1)^{m-2}(X^{2}-(-2+C(r,s))X+1)$.
\end{proof}
\subsection{Caractérisation du produit de deux réflexions qui est unipotent.}

\begin{proposition}\label{P unipotent}
On garde les hypothèses H(1) et les notations précédentes. Les conditions suivantes sont équivalentes:
\begin{enumerate}
  \item $rs$ est une application unipotente ($\neq Id_{M}$);
  \item $C(r,s)=4$;
  \item $H(r)\cap H(s)\,\cap <a,b>\neq {0}$;
  \item $H(r)$ = $H(s)$.
\end{enumerate}
\end{proposition}
\begin{proof}
D'après la proposition 2 il est clair que 1. et 2. sont équivalents. Comme $H(r)$ est un hyperplan, on a 4. $\Rightarrow$ 3..
Le sous-espace propre de $rs$ correspondant à la valeur propre $1$ est un hyperplan de $M$ lorsque $rs$ est unipotent ($\neq Id_{M}$), donc il est égal à $H(r)$ et à $H(s)$ d'où 1. $\Rightarrow$ 4.
Supposons la condition 3. satisfaite et soit 
\[
x=\alpha a+\beta b \in H(r)\cap H(s)\,\cap <a,b>\neq {0}
\]
On a:
\begin{equation*}
\begin{array}{ccccccc}
	r(x) & = & x & = & -\alpha a+\beta(b+c(r,a;s,b)a) & = & \alpha a+\beta b\\
	s(x) & = & x & = & \alpha(a+c(s,b;r,a)b)-\beta b & = & \alpha a+\beta b
	\end{array}
\end{equation*}
	
 On obtient le système suivant:

\begin{equation*}
\begin{array}{ccc}
2\alpha - \beta c(r,a;s,b) & = & 0\\
\alpha c(s,b;r,a) -2\beta & = & 0
\end{array}
\end{equation*}

et $x \neq 0$ si et seulement si le déterminant de ce système est nul, c'est  à dire si $-4+C(r,s)=0$, donc 3. $\Rightarrow$ 2.
\end{proof}
Nous étendons maintenant la définition de $C(r,s)$ au cas où $<a>$ = $<b>$ où $a$ (resp. $b$) est un vecteur directeur de $r$ (resp. $s$).
\begin{definition} Généralisation du coefficient de Cartan.\\
On garde les hypothèses de la proposition 1. Soient $r$ et $s$ deux réflexions de $M$, de vecteurs directeurs $a$ et $b$ respectivement. On suppose que $<a>$ = $<b>$: il existe $\lambda \in K^{*}$ tel que $b=\lambda a$. On a \\
$r(b)=-b=b-2b=b-2\lambda a$ et $s(a)=-a =a-2a=a-2\lambda^{-1}b$. On pose $C(r,s):=(-2\lambda)(-2\lambda^{-1})=4$:coefficient de Cartan du couple $(r,s)$.
\end{definition}
On a alors la caractérisation suivante:
\begin{proposition}
Avec les hypothèses de la proposition 1, soient $r$ et $s$ deux réflexions distinctes de $M$. Alors $rs$ est une application unipotente si et seulement si $C(r,s)=4$.
\end{proposition}
\begin{proof}
Soient $a$ et $b$ des vecteurs directeurs de $r$ et $s$ respectivement. On a déjà vu (proposition 1) le résultat lorsque $(a,b)$ est un système libre. On suppose donc que $<a>$ = $<b>$. $H(r)$ et $H(s)$ sont deux hyperplans distincts et $H(r)\cap H(s)$ est de codimension $2$ dans $M$. On peut donc supposer que $H(r)\cap H(s)={0}$, c'est à dire que $M$ est de dimension $2$, $H(r)$ et $H(s)$ étant alors des droites. Si $H(r)=<x>$ et $H(s)=<y>$, on a $M=<a,x>=<a,y>$ avec $y=\lambda x+\mu a$ et $\lambda\mu\neq0$. On obtient $s(y)= y=\lambda s(x)-\mu a=\lambda x+\mu a$ donc $s(x)=x+2\mu\lambda^{-1}a$. Dans la base $(x,a)$ de $M$ $rs(x)=x-2\mu\lambda^{-1}a$ et $rs(a)=a$, $rs$ est donc unipotente et d'après ce qui précède, $C(r,s)=4$.
\end{proof}
\subsection{Premières propriétés d'un groupe engendré par des réflexions}

Nous donnons d'abord quelques définitions dont nous aurons besoin dans toute la suite.\\
\begin{definition}
\begin{enumerate}
  \item Soit $G$ un groupe engendré par un ensemble $S$ d'involutions. on dit que $G$ est \textbf{$2$-sphérique} si $\forall(s,t)\in S\times S$, l'ordre $m_{st}$ de $st$ est fini.
  \item soient $G$ et $G'$ deux groupes $2$-sphériques engendrés par $S$ et $S'$ respectivement. Soit $R:G\to G'$ un morphisme tel que $R(S)\subset S'$. On dit que $R$ est un \textbf{bon} morphisme si $\forall (s,t)\in S\times S$, $R(s)R(t)$ a le même ordre que $st$.
\end{enumerate}
\end{definition}
Dans toute la suite de ce travail, $K$ est un corps de caractéristique $0$, $M$ est un $K$-espace vectoriel de dimension finie $n$ et $G$ est un sous-groupe de $GL(M)$ engendré par un ensemble $S$ de réflexions (on parle dans ce cas de \textbf{système de réflexion} $(G,S)$).

On définit le \textbf{graphe de Coxeter} de $(G,S)$, $\Gamma(G)$ de la manière suivante:
\begin{itemize}
  \item Les sommets de $\Gamma(G)$ sont les éléments de $S$;
  \item $(s,t)\in S\times S$ est une arête de $\Gamma(G)$ si $m_{st}\geqslant 3$, où $m_{st}$ est l'ordre de $st$;
  \item on décore chaque arête de $\Gamma(G)$ par le symbole $m_{st}$.
\end{itemize}
\begin{hypothesis}\hfill \label{H2}
\begin{enumerate}
  \item On a $|S|$ = $dim\,M$ ($=n$);
  \item $\forall(s,t)\in S\times S$, l'ordre $m_{st}$ de $st$ est fini;
  \item le graphe $\Gamma(G)$ est connexe;
  \item si $s \in S$ on appelle $a_{s}$ un vecteur directeur de $s$; alors $\mathcal{A}:=(a_{s}|s\in S)$ est une base de $M$
\end{enumerate}

\end{hypothesis}
Soit $(G,S)$ un système de réflexion. On pose $\textbf{T}:=\{gsg^{-1}|g\in G, s\in S\}$ ensemble des réflexions du système $(G,S)$.

On peut remarquer que toutes les réflexions de $G$ sont de déterminant $-1$, donc $G$ contient un sous-groupe, noté $G^{+}$, formé des éléments de $G$ de déterminant $+1$, ou encore qui sont produits d'un nombre pair d'éléments de $S$:
$G^{+}=G\cap SL(M)$. Il est clair que $G^{+}$ est d'indice $2$ dans $G$. On a toujours $D(G) (=[G,G])\subset G^{+}$.

Comme $G/D(G)$ est engendré par les images des éléments de $S$, on voit que c'est un $2$-groupe commutatif élémentaire d'ordre $2^{\omega}$, où $\omega$ est le nombre des classes de conjugaison contenues dans $\textbf{T}$.
\begin{proposition}\hfill \label{P5}
\begin{enumerate}
  \item $C_{GL(M)}(G)$ est formé d'applications scalaires;
  \item soit $z\in Z(G)$. Si $\det z=1$, l'ordre de $z$ est un diviseur de $n$; si $\det z=-1$, l'ordre de $z$ est un diviseur de $2n$;
  \item si $C_{M}(G) \neq {0}$, on a $Z(G) =1$.
\end{enumerate}

\end{proposition}
\begin{proof}
1) Soit $z\in C_{GL(M)}(G)$. Pour tout $t\in \textbf{T}$, si $b$ est un vecteur directeur de $t$, $z(b)=\lambda_{t}b$ avec $\lambda_{t} \in K^{*}$. Soit $e=(s,t)\in E(\Gamma(G))$. On a $s(a_{t})=a_{t}+\mu_{s}a_{s}$ , $t(a_{s})=a_{s}+\mu_{t}a_{t}$ et 
$sts(a_{s})=a_{s}-\mu_{t}(a_{t}+\mu_{s}a_{s})$ avec $\mu_{s}\mu_{t}\neq 0$. On obtient $z(a_{s})=\lambda_{s}a_{s}$, $z(a_{t})=\lambda_{t}a_{t}$ et $z(a_{t}+\mu_{s}a_{s})=\lambda_{sts}(a_{t}+\mu_{s}a_{s})$; mais $z(a_{t}+\mu_{s}a_{s}) = \lambda_{t}a_{t}+\mu_{s}\lambda_{s}a_{s}$. Comme $(a_{s},a_{t})$ est un système libre et comme $\mu_{s}\neq 0$, on obtient $\lambda_{sts}=\lambda_{s}=\lambda_{t}$. Comme $\Gamma(G)$ est connexe, on a le résultat: $\exists \lambda \in K^{*}$ tel que $\forall m \in M, z(m)=\lambda m$ puisque l'ensemble $(a_{s}|s \in S)$ est une base de $M$: $z$ est une application scalaire.\\
2) On a alors $\det z = \lambda^{n}$. Si $\det z = 1$, $\lambda^{n}=1$ et l'ordre de $z$ est un diviseur de $n$; si $\det z = -1$, l'ordre de $z$ est un diviseur de $2n$.\\
3) Si $m \in C_{M}(G)-{0}$, on a $z(m)=m=\lambda m$ donc $\lambda =1$ et $z=1$: $Z(G)={1}$.

\end{proof}
On définit maintenant la matrice de Cartan.
On garde les hypothèses et notations précédentes.\\ On suppose que $S$ = $\{s_{1},s_{2},\cdots,s_{n}\}$ et que pour tout $i$, $a_{i}$ est un vecteur directeur de $s_{i}$. On a $\forall(i,j) (1\leqslant i,j \leqslant n),s_{i}(a_{j})=a_{j}-\lambda _{ij}a_{i}$ où $\lambda_{ij}=-C(s_{i},s_{j})$ et $\lambda_{ij}=\lambda_{ji}=0$ si $s_{i}s_{j}=s_{j}s_{i}\neq 1$.
\begin{definition}\label{D3}
On appelle \textbf{matrice de Cartan} de $G$ (par rapport à la base $\mathcal{A}$ de $M$) la matrice 
\[
Car(G)=(\lambda_{ij})_{1\leqslant i,j\leqslant n}.
\]

\end{definition}
\section{La construction fondamentale} 
Dans toute la suite de ce travail on fait les hypothèses suivantes H(Cox) sur le système de Coxeter $(W,S)$:
\textbf{$(W,S)$ est de rang fini, $2$-sphérique et irréductible}.

Soit $(W,S)$ un système de Coxeter. Le but de cette section est de construire des représentations $R \to GL(M)$ avec $M$ $K$-espace vectoriel de $W$, qui sont telles que les éléments $R(s)$ soient des réflexions de $M$ et $\forall(s,t)\in S\times S,R(s)R(t)$ a le même ordre que $st$.
\begin{notation}
On pose $G:=Im R$.
\end{notation}
On désignera par la même lettre $g\in W$ et $R(g)$ lorsqu'il n'y aura pas de confusion. \\
On appelle $K_{0}$ un corps de décomposition de l'ensemble des polynômes $v_{m_{e}}(X)$ (si $e=(s,t), m_{e}=m_{st})$.
On peut remarquer que l'on peut choisir $K_{0} \subset \mathbb{R}$ et alors $K_{0}$ est un sous-corps réel d'un corps cyclotomique. \textbf{Dans la suite, on fera toujours ce choix pour $K_{0}$.}
\subsection{Quelque résultats de la théorie des graphes}
\begin{proposition}
Soit $\Gamma$ un graphe fini simple (sans boucle ni arête multiple) connexe et soient $\textit{T}$ un \textbf{arbre couvrant} de $\Gamma$ et $s_{0}$ un sommet de $\Gamma$ (appelé \textbf{racine} de \textit{T}). On appelle $E(\Gamma)$ l'ensemble des arêtes de $\Gamma$, $E(\textit{T})$ l'ensemble des arêtes de $\textit{T}$ et on pose $E'(\textit{T}):=E(\Gamma)-E(\textit{T})$.\\
	1) On définit sur l'ensemble des sommets de $\Gamma$ une relation d'ordre, notée $\preccurlyeq$ (qui dépend de $\textit{T}$ et de $s_{0}$) de la manière suivante: $s\preccurlyeq t$ si $s$ et $t$ sont sur la même branche de l'arbre $\textit{T}$ et si la distance de $s$ à $s_{0}$ (dans $\textit{T}$) est $\leqslant $ à la distance de $t$ à $s_{0}$ (dans $\textit{T}$).\\
	2) Soit $e:=(s,t)\in E'(\textit{T})$. Lorsque l'on adjoint cette arête à $\textit{T}$, on obtient un unique circuit
\[
C \,(=C(e)):=\{s_{1}=s,s_{2},\cdots, s_{n-1},s_{n}=t\}
\]
avec $(s_{i},s_{i+1})\in E(\textit{T})\, (1\leqslant i \leqslant n-1)$. Alors il existe un unique $p$ $(1\leqslant p \leqslant n)$ tel que $\forall i\, (1\leqslant p \leqslant n)$ on ait $s_{p} \preccurlyeq s_{i}$. On dit que $s_{p}$ est \textbf{l'entrée} dans le circuit $C$.
\end{proposition}
\begin{proof}

\end{proof}
\begin{proposition}[\"Ore]\label{Po}
Soient $\textit{T}$ et $\textit{T'}$ deux arbres couvrants d'un graphe fini, simple, connexe $\Gamma$. Alors on peut passer de $\textit{T}$ à $\textit{T'}$ par une suite finie d'opérations de la forme:\\
On ajoute une arête à $\textit{T}$ pour obtenir un cycle, puis on enlève un autre arête de ce cycle pour obtenir un arbre couvrant. On obtient $\textit{T'}$ après une suite finie de ces opérations. (voir  \cite{O}, théorème 6.4.4).
\end{proposition}
\subsection{La construction fondamentale}
On considère le graphe  $\Gamma(G)$.

Soient $K'$ un sur-corps de $K_{0}$ et $M'$ un $K'$-espace vectoriel. On va construire des représentations $R \to GL(M')$ qui satisfont aux conditions H(R) suivantes:
\begin{enumerate}
  \item  $\forall s \in S$, $R(s)$ est une réflexion de $M$;
  \item Si $s\in S$ et si $a_{s}$ est un vecteur directeur de $R(s)$, alors $\mathcal{A}:=\{a_{s}|s\in S\}$ est une base de $M$;
  \item  $\forall (s,t)\in S\times S$, $R(s)R(t)$ a le même ordre que $st$.
\end{enumerate}
Pour cela, nous allons faire une série de choix.\\
On choisit un arbre couvrant $\textit{T}$ et $s_{0}$ un sommet (une racine) de $\Gamma(G)$ . On appelle $E(\Gamma(G))$ l'ensemble des arêtes de $\Gamma(G)$ et $E(\textit{T})$ l'ensemble des arêtes de $\textit{T}$, enfin on pose $E'(\textit{T}) :=E(\Gamma(G))-E(\textit{T})$.\\
On définit sur l'ensemble des sommets de $\Gamma(G)$ une relation d'ordre, notée $\preccurlyeq$ comme dans la proposition 6.\\
On définit des éléments $\zeta_{s}$ ($s\in S$) de $GL(M)$ de la manière suivante:
\begin{itemize}
  \item $\forall s \in S, \zeta_{s}(a_{s})=-a_{s}$;
  \item soit $(s,t)\in S^{^2}$ tel que $st=ts\neq 1$, on pose $\zeta_{t}(a_{s})=a_{s}$;
  \item soit $e:=(s,t)\in E(\textit{T})$ avec $s\preccurlyeq t$, on pose $\zeta_{s}(a_{t})=\alpha_{e}a_{s}+a_{t}$ et $\zeta_{t}(a_{s})=a_{s}+a_{t}$;
  \item soit $e:=(s,t)\in E'_{\textit{T}}$, on pose $\zeta_{s}(a_{t})=l_{e}a_{s} +a_{t}$ et $\zeta_{t}(a_{s})=l'_{e}a_{t} +a_{s}$.
\end{itemize}
\begin{notation}
On appelle $K:=K_{0}(l_{e}| e\in E'(\textit{T}))$ le sous-corps de $K'$ engendré par $K_{0}$ et tous les $l_{e}$ ($e\in E'(\textit{T})$) et on pose $M:=M'\otimes_{K_{0}}K$.
\end{notation}
L'application $R :s \mapsto \zeta_{s}: S \to {\zeta_{s} | s \in S}$ se prolonge en une représentation (notée aussi $R$) $R: W \to GL(M)$ qui possède par construction les propriétés H(R). On dit que $R$ est obtenue par la \textbf{construction fondamentale}.

On dit que $\mathcal{A}:=\{a_{s}|s\in S\}$ est une \textbf{base adaptée} à la construction fondamentale, toute autre base s'obtient en multipliant chaque élément de $\mathcal{A}$ par un scalaire non nul, toutes les autres bases adaptées s'obtiennent en multipliant tous les éléments de $\mathcal{A}$ par un même scalaire.
\begin{definition}\label{paramètres}
On pose $\mathcal{P}(G)$ = $\mathcal{P}((\alpha_{e}| e\in E(\textit{T})) \cup (l_{e}| e \in E'_{\textit{T}}))$ et on appelle $\mathcal{P}(G)$ le \textbf{système de paramètres} de $G$.
\end{definition}
On peut remarquer que si $G$ et $G'$ ont le même système de paramètres alors ils sont isomorphes; de plus tous les éléments des matrices des éléments de $Im R$ s'écrivent dans l'anneau $\mathcal{O}(K)$ sous-anneau de $K$ engendré par les éléments du système de paramètres.
\begin{theorem}[Théorème fondamental]
Soit $R$ une représentation de réflexion du groupe de Coxeter $W$. Alors:
\begin{enumerate}
  \item Si l'on change d'arbre couvrant ou bien si l'on change de racines, on obtient une représentation équivalente.
  \item Si l'on change de système de paramètres, on obtient une représentation inéquivalente.
 
\end{enumerate}

\end{theorem}
La démonstration va occuper le reste de cette section.
\begin{proposition}\label{P8}
Si dans la construction fondamentale, on change la racine de l'arbre, on obtient une représentation équivalente.
\end{proposition}
\begin{proof}
Comme le graphe $\Gamma(G)$ est connexe, on peut choisir $s_{1}\in S$ tel que $e:=(s_{0},s_{1})\in E(\textit{T})$. On pose pour simplifier la notation $\alpha := \alpha_{m_{e}}$. On pose $a'_{s}:=\lambda_{s}a_{s}$ avec $\lambda_{s}\in K^{*} (s\in S)$ et $\lambda_{1}=1$. On a $a'_{1}=a_{1}$. On appelle $R'$ la représentation de $W$ ainsi obtenue.\\
On définit $\Gamma_{i}\quad (i=1,2)$ par:
\begin{itemize}
  \item $\Gamma_{1}:=\{s |s\in S$\}, la distance de $s$ à $ s_{0}$\ dans $\textit{T}$ est $\leqslant $ à la distance de $ s$  à $ s_{1}$ dans $ \textit{T} $,
  \item $\Gamma_{2}:=S-\Gamma_{1}$.
  
\end{itemize}
D'après la construction fondamentale avec $s_{1}$ comme racine on a:
\begin{align}
s_{0}(a_{1}) & =a_{1}+a'_{0} \notag\\
s_{1}(a'_{0}) &=a'_{0}+\alpha a_{1} \notag
\end{align}
mais $a'_{0}$ = $\lambda_{0}a_{0}$ et $s_{0}(a_{1})$ = $a_{1}+\alpha a_{0}$, donc $\lambda_{0}$ = $\alpha$.
\begin{itemize}
  \item Soit $t\in \Gamma_{1}-\{s_{1}\}$ tel que  $(t,s_{0})\in E(\textit{T})$, alors  $a'_{t}=t(a'_{0})-a'_{0}=\alpha (t(a_{0})-a_{0})=\alpha a_{t}=\lambda_{t}a_{t}$, donc $\lambda_{t}=\alpha$. En parcourant la partie de l'arbre qui est dans $\Gamma_{1}$, on voit que $\forall s \in \Gamma_{1}$, $\lambda_{s}=\alpha$.
  \item Soit $t\in \Gamma_{2}$ tel que  $(t,s_{1})\in E(\textit{T})$. On a $a'_{t}=t(a_{1})-a_{1}=a_{t}$, donc \\$\lambda_{t} =1$. En parcourant la partie de l'arbre qui est dans $\Gamma_{2}$, on voit que $\forall s \in \Gamma_{2}$, $\lambda_{s}=1$.
\end{itemize}
Soit maintenant $g\in GL(M)$ défini de la manière suivante : si $s\in \Gamma_{1}$ on pose $g(a_{s}) :=\alpha a_{s}$, si $s\in \Gamma_{2}$ on pose $g(a_{s})=a_{s}$. Alors $g$ est un opérateur d'entrelacement entre les représentations $R$ et $R'$.\\ En répétant ce raisonnement avec tous les éléments de $S$, on obtient le résultat.
\end{proof}
\subsection{Changement d'arbres}
\begin{proposition} \label{P9}
Si dans la construction fondamentale on change d'arbre couvrant, on obtient une représentation équivalente.
\end{proposition}
\begin{proof}
D'après la proposition 8 on peut choisir la racine de l'arbre comme l'on veut. On suppose que $C:=\{s_{1},s_{2},\cdots, s_{n}\}$ est un circuit de $\Gamma(G)$ avec $s_{1}$ racine de $\textit{T}$ , les arêtes $(s_{i},s_{i+1})$ $(1\leqslant i \leqslant n-1)$ sont dans $E(\textit{T})$ et $(s_{1},s_{n})\in E'(\textit{T})$. On enlève l'arête $e_{m}:=(s_{m},s_{m+1})$ pour obtenir un autre arbre couvrant $\textit{T'}$, et on obtient ainsi la représentation $R'$ de $W$.\\
Correspondant à $\textit{T}$ on a la base de $M$ : $\emph{A}=(a_{1},a_{2},\cdots,a_{n})\cup \cdots$ et correspondant à $\textit{T'}$ on a la base de $M$ : $\emph{A'}=(a'_{1},a'_{2},\cdots,a'_{n})\cup \cdots$ avec $a'_{s}=\lambda_{s}a_{s}$ $(s\in S)$ et $\lambda_{1}=1$: $a'_{1}=a_{1}$. On a $l_{1n}l_{n1}=\alpha_{1n}\,(=\alpha_{n1})$ et $l'_{m,m+1}l'_{m+1,m}=\alpha_{m,m+1}\,(=\alpha_{m+1,m})$.\\ 
Les branches $T'_{j}$ de $\textit{T'}$ qui passent par $s_{j}$ $(1\leqslant j \leqslant m)$ sont les mêmes que les branches $T_{j}$ de $\textit{T}$ qui passent par $s_{j}$.

	\textbf{ Pour $1\leqslant j \leqslant m$, on a $\lambda_{j}=1$ et $\forall s\in T'_{j}$, $a'_{s}=a_{s}$}.
	\begin{proof}
On procède par récurrence sur $j$. le résultat es vrai par hypothèse pour $j=1$. Soit $j\geqslant 2$ et le résultat supposé vrai pour $j-1$. Alors 
\[
s_{j}(a'_{j-1})=s_{j}(a_{j-1})=a'_{j-1}+a'_{j}=a_{j-1}+\lambda a_{j}=a_{j-1}+a_{j}
\]
donc $\lambda_{j}=1$ et l'on a le résultat. Soit $s \in T'_{j}$ $(1\leqslant j \leqslant m)$, alors par le même raisonnement, on voit que $a'_{s}=a_{s}$.
\end{proof}

\textbf{Pour $m+1 \leqslant j \leqslant n$, on a $\lambda_{n}=l_{n,1}$ et $\lambda_{j}=\alpha_{j,j+1}\alpha_{j+1,j+2}\cdots \alpha_{n-1,n}l_{n,1}$ et $\forall s\in T'_{j}$ $(m+1\leqslant j \leqslant n)$, on a $a'_{s}=\lambda_{j}a_{s}$}.
\begin{proof}
On procède par récurrence descendante sur $j$. On a
\[
s_{n}(a'_{1})=s_{n}(a_{1})=a_{1}+l_{n,1}a_{n}=a'_{1}+a'_{n}=a_{1}+\lambda_{n}a_{n},
\]
donc $\lambda_{n}$ = $l_{n,1}$.\\
Soit $j < n$. On suppose le résultat vrai pour $j+1$: $\lambda_{j+1}=\alpha_{j+1,j+2}\cdots \alpha_{n-1,n}l_{n,1}$, alors
\[
s_{j}(a'_{j+1})=\lambda_{j+1}\alpha_{j,j+1}a_{j}+\lambda_{j+1}a_{j+1}=\lambda_{j}a_{j}+\lambda_{j+1}a_{j+1},
\]
donc $\lambda_{j}=\alpha_{j,j+1}\lambda_{j+1}$, d'où le résultat. Si $s\in T'_{j}$, $(m+1 \leqslant j \leqslant n)$, alors par le même raisonnement, on voit que $a'_{1}=\lambda_{j}a_{1}$.
\end{proof}
\textbf{Fin de la démonstration.}\\
On définit $g\in GL(M)$ de la manière suivante:
\begin{itemize}
  \item pour $1\leqslant j \leqslant m$, $g(a_{j})=a_{j}$ et $\forall s \in T_{j}$, $g(a_{s})=a_{s}$;
  \item pour $m+1\leqslant j \leqslant n$, $g(a_{j})=\lambda_{j}a_{j}$ et $\forall s \in T_{j}$, $g(a_{s})=\lambda_{j}a_{s}$.
  \end{itemize}
  Alors $g$ est un opérateur d'entrelacement entre les représentations $R(T)$ et $R(T')$ de $W$.\\
  Si $T"$ est un arbre couvrant quelconque de $\Gamma(G)$, on utilise la proposition 7.
\end{proof}
\begin{proposition}\label{P10}
Soit $(W,S)$ un système de Coxeter qui satisfait aux hypothèses H(Cox) et soit $R$ et $R'$ deux représentations de réflexion de $W$ obtenues grâce à la construction fondamentale. On pose $G:=ImR$ et $G':=ImR'$. Si les systèmes de paramètres sont distincts, alors $R$ et $R'$ ne sont pas équivalentes.
\end{proposition}
\begin{proof}
On peut supposer que l'on a les mêmes arbres couvrants et racines pour $R$ et $R'$. Pour montrer que ces deux représentations ne sont pas équivalentes, on montre que leurs caractères sont distincts.

Soit $e:=(s,t)$ une arête de $\textit{T}$. D'après la proposition 2, on a $tr(st)=|S|-4+C(s,t)$ avec $C(s,t)=\alpha_{e}$. Si $\alpha_{e}\neq \alpha_{e'}$, les traces sont distinctes et on a le résultat dans ce cas.

Soit maintenant $e:=(s_{1},s_{m}) \in E'(\textit{T})$. On a $l_{e}=l_{1,m}$ et $l_{e'}=l_{m,1}$. Si l'on ajoute cette arête à $\textit{T}$, on obtient un circuit $C:=\{s_{1},s_{2},\cdots,s_{m}\}$ avec pour $1\leqslant i\leqslant m-1$, $(s_{i},s_{i+1})\in E(T)$.\\
On pose $u:=s_{1}s_{2}\cdots s_{m}$ et on montre que $tr(u)$ est un polynôme du premier degré en $l_{e'}$ à coefficients dans $K_{0}$ ou $K$, ce qui donne le résultat.\\
Comme on ne cherche que la trace de $u$, on n'écrira en général dans les calculs qui suivent que le coefficient de $a_{s}$ $(s\in S)$ dans $u(a_{s})$. On distingue deux cas suivant que dans $C$, sous-graphe de $\Gamma(G)$, il y a des cordes ou pas.

\textbf{Premier cas}:  Dans $C$ il n'y a pas de cordes: les seules arêtes de $C$ contenant $s_{i}$ sont $(s_{i-1},s_{i})$ et $(s_{i},s_{i+1})$ $(2\leqslant i\leqslant m \,\text{avec} \,s_{m+1}=s_{1})$. On montre que dans ces conditions, $tr(u)$ est un polynôme du premier degré en $l_{e'}$ à coefficients dans $K_{0}$.\\
Nous pouvons déjà remarquer (et cela sera aussi valable dans le deuxième cas) que si $s\in S-C$, alors $u(a_{s})=a_{s}+\cdots$, donc en faisant ceci pour tous les éléments de $S-C$, on voit qu'ils contribuent $|S|-m$ à la trace de $u$.\\
On a:
\[
u(a_{1})=s_{1}s_{2}\cdots s_{m-1}(a_{1}+l_{e'}a_{m})=s_{1}s_{2}(a_{1})+l_{e'}s_{1}s_{2}\cdots s_{m-1}(a_{m})
\]
On a:
\begin{eqnarray*}
s_{1}s_{2}\cdots s_{m-1}(a_{m})&=&s_{1}s_{2}\cdots s_{m-2}(\alpha_{m-1,m}a_{m-1}+a_{m})\\
&=&\alpha_{m-1,m}s_{1}s_{2}\cdots s_{m-2}(a_{m-1})+s_{1}(a_{m})
\end{eqnarray*}
mais $s_{1}(a_{m})=l_{e}a_{1}+a_{m}$, $s_{1}s_{2}(a_{1})=(\alpha_{1,2}-1)a_{1}+\cdots$, puis par une récurrence facile, on trouve:
\[
u(a_{1})=((\alpha_{1,2}-1)+\alpha_{1,m}+\alpha_{1,2}\alpha_{2,3}\cdots\alpha_{m-1,m}l_{m-1,1})a_{1}+\cdots
\]
pour $2\leqslant k \leqslant m-1$, on a: $u(a_{k})=(\alpha_{k,k-1}-1)a_{k}+\cdots$; $u(a_{m})=-a_{m}+\cdots$.\\
Il en résulte que l'on a:
\[
tr(u) =|S|-m+\sum_{k=1}^{m}(\alpha_{k,k+1}-1)+(\prod_{k=1}^{m-1}\alpha_{k,k+1})l_{e'} 
\]
\[
tr(u)=|S|-2m+\sum_{k=1}^{m}\alpha_{k,k+1}+(\prod_{k=1}^{m-1}\alpha_{k,k+1})l_{e'}
\]
(les indices ($\mod{m}$)) et le résultat dans ce cas.\\
\textbf{Deuxième cas}: Dans $C$ il y a des cordes: il existe $p$ et $q$ tels que $1\leqslant p<q \leqslant m$ et $(s_{p},s_{q})\in E'(T)$. Nous choisissons $p$ et $q$ de telle sorte que le circuit 
\[
C':=\{s_{1},\cdots ,s_{p-1},s_{p},s_{q},s_{q+1},\cdots,s_{m}\}
\]
 n'ait pas de cordes avec des variantes si $p=1$ ou si $q=m$. Cela est toujours possible car le graphe $\Gamma(W)$ est fini.\\ Nous faisons le même calcul que dans le premier cas et nous voyons que si
 \[
 v:=s_{1}\cdots s_{p-1}s_{p}s_{q}s_{q+1}\cdots s_{m}
 \]
 alors $tr(v)$ est un polynôme du premier degré en $l_{e'}$ à coefficients dans $K$ et le résultat dans ce cas.
  \end{proof}
  \begin{proof}
Ceci termine la démonstration du théorème fondamental.
\end{proof}
\subsection{La représentation géométrique}
Nous calculons  maintenant les paramètres de la représentation géométrique. 

Soit $(W,S)$ un système de Coxeter satisfaisant aux hypothèses H(Cox). Dans la construction fondamentale, on a le corps $K$, pour chaque arête $e$ de $\Gamma(W)$, $\alpha_{e}$ une racine de $v_{e}(X)$ et enfin pour chaque arête de $E'(\textit{T})$ un élément $l_{e}$ de $K$ et $K=K_{0}(l_{e}\,|\, e\in E'(\textit{T}) )$.

La représentation géométrique est dans la base $\emph{B}:=(b_{s}=\lambda_{s}a_{s}\,|\,s \in S, \lambda_{s}\in K^{*})$ où $K=K_{0}(\cos \frac {\pi}{m_{st}}\, |\, (s,t)\in S^{2})$ et l'on a, si $(s,t)\in S^{2},\, t\neq s$:
\begin{itemize}
  \item $t(b_{s})=b_{s}$ si $st=ts$;
  \item $t(b_{t})=-b_{t}$;
  \item si $m_{st}\geqslant 3$ $t(b_{s})=b_{s}+2\cos \frac {\pi}{m_{st}}b_{t}$ et $s(b_{t})=b_{t}+2\cos \frac {\pi}{m_{st}}b_{s}$.
\end{itemize}

Dans ce qui suit, nous allons déterminer les $\lambda_{s}$, $(s\in S)$, $\alpha_{e}$, $(e\in E(\Gamma(W)))$ et les $l_{e}$, $(e\in E'(T))$.

Dans la construction fondamentale, nous avons un arbre couvrant $\textit{T}$ et une racine $s_{0}$ de cet arbre.\\
	1) Soit $(s,t)\in E(T)$ avec $s\preccurlyeq t$. Alors on a 
\[
t(b_{s})=\lambda_{s}a_{s}+\lambda_{t}2\cos \frac{\pi}{m_{st}}a_{t}=t(\lambda_{s}a_{s})=\lambda_{s}(a_{s}+a_{t})
\]
donc $\lambda_{s} =\lambda_{t}\,2\cos \frac{\pi}{m_{st}}$;
\[
s(b_{t})=\lambda_{t}a_{t}+\lambda_{s}\,2\cos \frac{\pi}{m_{st}}=s(\lambda_{t}a_{t})=\lambda_{t}a_{t}+\lambda_{t}\alpha_{st}a_{s}
\]
donc $\alpha_{st}\lambda_{t}=\lambda_{s}\,2\cos \frac{\pi}{m_{st}}$. Il en résulte que $\alpha_{st}=4\cos^{2}\frac{\pi}{m_{st}}$.\\
	2) Soit $(s,t)\in E'(T)$. On a:
\[
t(b_{s})=\lambda_{s}a_{s}+\lambda_{t}\,2\cos \frac{\pi}{m_{st}}a_{t}=t(\lambda_{s}a_{s})=\lambda_{s}a_{s}+\lambda_{s}l_{ts}a_{t}
\]
donc $\lambda_{t}\,2\cos \frac{\pi}{m_{st}}=\lambda_{t}l_{ts}$. Nous obtenons $\alpha_{st}=l_{st}l_{ts}=4\cos^{2}\frac{\pi}{m_{st}}$.\\
Les $\alpha_{st}$ sont donc entièrement déterminés indépendamment  de $\textit{T}$ et de $s_{0}$.

Il nous reste à déterminer les $\lambda_{s}$, ($s\in S$) et les $l_{e}$, ($e\in E'(T)$).\\
3) Si l'on ajoute l'arête $e\in E'(T)$ à $\textit{T}$ on obtient un circuit 
\[
C:=(s_{1}(=s),s_{2},\cdots,s_{n-1}(=t)).
\]
On appelle $p$ l'entrée dans $C$. On étudie d'abord le cas où $2\leqslant p\leqslant n-1$, puis nous ferons les modifications nécessaires lorsque $p=1$ ou $p=n$.

Pour simplifier les notations  on pose \\$B':=\prod_{k=1}^{n-1}2\cos \frac{\pi}{m_{k,k+1}}$ et $B:=B'2\cos \frac{\pi}{m_{n,n+1}}$ et aussi $\alpha_{k}:=\alpha_{k,k+1}$, les indices étant pris $\pmod{n}$.\\
Soit $q\in \{1,2,\cdots n\}$.\\
- Si $q \geqslant p$ alors $s_{q}\preccurlyeq s_{q+1}$ donc on a, d'après le 1), $\lambda_{q}=2\cos \frac{\pi}{m_{q,q+1}}\lambda_{q+1}$ et on obtient pour $p\leqslant q \leqslant n-1$, $\lambda_{q}=(\prod_{k=q}^{n-1}2\cos \frac{\pi}{m_{q,q+1}})\lambda_{n}$.\\ En particulier $\lambda_{p}=(\prod_{k=p}^{n-1}2\cos \frac{\pi}{m_{q,q+1}})\lambda_{n}$.

-Si $2\leqslant q \leqslant p$ alors $s_{q}\preccurlyeq s_{q-1}$ donc on a, d'après le 1), $\lambda_{q}=(\prod_{k=2}^{q}2\cos \frac{\pi}{m_{k,k+1}})\lambda_{1}$. En particulier $\lambda_{p}=(\prod_{k=2}^{p}2\cos \frac{\pi}{m_{k,k-1}})\lambda_{1}$ = $(\prod_{k=1}^{p-1}2\cos \frac{\pi}{m_{k,k+1}})\lambda_{1}$. On obtient ainsi 
\[
(\prod_{k=p}^{n-1}2\cos \frac{\pi}{m_{k,k+1}})\lambda_{n}=(\prod_{k=1}^{p-1}2\cos \frac{\pi}{m_{k,k+1}})\lambda_{1}.
\]
Nous avons vu au 2) que $\lambda_{n}2\cos \frac{\pi}{m_{n,n+1}}=l_{m,1}\lambda_{1}$ et $\lambda_{1}2\cos \frac{\pi}{m_{n,n+1}}=l_{1,n}\lambda_{n}$. Nous obtenons alors: 
\[
\lambda_{n}B'2\cos\frac{\pi}{m_{n,n+1}}=\lambda_{1}(\prod_{k=1}^{p-1}\alpha_{k})2\cos\frac{\pi}{m_{n,n+1}}=\lambda_{1}(\prod_{k=1}^{p-1}\alpha_{k})l_{1,n}
\]
d'où $(\prod_{k=1}^{p-1}\alpha_{k})l_{1,n}=B$. De la même manière $(\prod_{k=p}^{n-1}\alpha_{k})l_{n,1}=B$.

Ceci nous donne les valeurs de $l_{1,n}$, $l_{n,1}$ et des $\lambda_{p}$ pour tout $p$ en fonction de $\lambda_{1}$.\\
Si $p=1$, alors $\prod_{k=1}^{p-1}\alpha_{k}=1$ par convention et on obtient $l_{1,n}$ = $B$ et $\alpha_{n}$ = $Bl_{n,1}$.\\
Si $p=n$, alors $\prod_{k=p}^{n-1}\alpha_{k}=1$ par convention et on obtient $l_{n,1}$ = $B$ et $\alpha_{n}$ = $Bl_{1,n}$.

Nous procédons de la même manière pour toutes les arêtes $e\in E'(T)$ et nous obtenons les valeurs de tous les paramètres.
\subsection{Formes sesquilinéaires invariantes}
Nous nous intéressants maintenant aux formes sesquilinéaires $G$-invariantes.

Soit $(W,S)$ un système de Coxeter satisfaisant aux hypothèses H(Cox) et soit $R$ une représentation de réflexion de $W$ obtenue grâce à la construction fondamentale (arbre couvrant $\textit{T}$ et racine $s_{0}$). On a posé, comme d'habitude $G=Im R$. Soit $\theta \in Aut \,K$. On cherche les formes $\theta$-sesquilinéaires  $G$-invariantes.\\
Soit $\Phi:=\{\varphi | \varphi : M\times M \to K\}$, $\varphi$ forme $\theta$-sesquilinéaire $G$-invariante ($\varphi $ est linéaire par rapport à la première variable). Il est clair que $\Phi$ est un $K$-espace vectoriel.

Si $s\in S$, il existe un unique chemin dans $\textit{T}$ qui joint $s$ à $s_{0}$ $(s_{0},s_{1},\cdots,s_{p}=s)$. Pour $0\leqslant i \leqslant p-1$, on a une racine $\alpha_{i}$ de $v_{m_{s_{i},s_{i+1}}}(X)$. On pose $\prod_{\textit{T},s_{0}}(s):=\prod_{i=0}^{p-1}\alpha_{i}$. On a alors le théorème:
\begin{theorem}\hfill\label{T3}
\begin{enumerate}
  \item On a $dim \,\Phi \leqslant 1$;
  \item les deux assertions suivantes sont équivalentes:
  \begin{enumerate}
  \item $dim\,\Phi =1$
  \item \begin{enumerate}
  \item $\theta^{2}=id_{K}$;
  \item $K_{0}$ est contenu dans le sous-corps des points fixes de $\theta$;
  \item soit $e:=(s,t)\in E'(T)$; alors $e$ est contenu dans un unique circuit $C(e)=\{s_{1}=s,s_{2}\cdots, s_{q-1},s_{q}=t\}$ avec $(s_{i},s_{i+1})\in E(T)$ $(1\leqslant i \leqslant q-1)$. On a 
  \[
  \theta(l_{st})\prod_{\textit{T},s_{0}}(s)=l_{ts}\prod_{\textit{T},s_{0}}(t).
  \]

\end{enumerate}
 
\end{enumerate}
\end{enumerate} 
 Si ces conditions sont satisfaites, $\varphi$ est une forme $\theta$-hermitienne.
\end{theorem}
\begin{proof}
On va supposer qu'il existe une forme $\theta$-sesquilinéaire non nulle et $G$-invariante $\varphi : M\times M \to K$. Pour tout $(s,t)\in S^{2}$, on pose $\beta_{st}:=\varphi(a_{s},a_{t})$ et $\beta_{t}:=\beta_{tt}$.

On procède en quelques étapes.

\textbf{Etape 1}. Si $s$ et $t$ dans $S$ commutent et sont distincts, on a $\beta_{st}=0$.
\begin{proof}
$\beta_{st}=\varphi(a_{s},a_{t})=\varphi(s(a_{s}),s(a_{t}))=\varphi(-a_{s},a_{t})=-\beta_{st}$. Comme $K$ est de caractéristique $0$, on obtient $\beta_{st}=0$.
\end{proof}
\textbf{Etape 2}. Soit $(s,t)\in E(T)$ avec $s\preccurlyeq t$. Alors si $\beta_{st}\neq 0$ on a : $\beta_{t}=-2\beta_{st}$, $\beta_{ts}=\beta_{st}$, $\alpha_{st}\beta_{s}=-2\beta_{st}$ et $\theta(\alpha_{st})=\alpha_{st}$.
\begin{proof}
On a:
\[
\beta_{st}=\varphi(a_{s},a_{t})=\varphi(s(a_{s}),s(a_{t}))=\varphi(-a_{s},\alpha_{st}a_{s}+a_{t})=-\beta_{st}-\theta(\alpha_{st})\beta_{s},
\]
d'où $\theta(\alpha_{st})\beta_{s}=-2\beta_{st}$. On a aussi:
\[
\beta_{st}=\varphi(a_{s},a_{t})=\varphi(t(a_{s}),t(a_{t}))=\varphi(a_{s}+a_{t},-a_{t})=-\beta_{st}-\beta_{t}
\]
d'où $\beta_{t}=-2\beta_{st}$. On a par un calcul semblable: $\beta_{ts}=-\alpha_{st}\beta_{s}-\beta_{ts}$, $\alpha_{st}\beta_{s}=-2\beta_{ts}$ et aussi $\beta_{ts}=-\beta_{ts}-\beta_{t}$, d'où $\beta_{t}=-2\beta_{ts}$.\\
Nous voyons ainsi que $\beta_{ts}=\beta_{st}$ et si $\beta_{st}\neq 0$, $\theta(\alpha_{st})=\alpha_{st}$. Il en résulte que $\alpha_{st}$ est dans le sous-corps des points fixes de $\theta$.
\end{proof}
\textbf{Etape 3}. Soit $(s_{0},s_{1},\cdots s_{p})$ une branche de $\textit{T}$. Alors il existe $\mu \in K$ tel que:
\[
\beta_{s_{i}}=2\mu \prod_{\textit{T},s_{0}}(s_{i})\quad (0\leqslant i \leqslant p)\quad \text{et} \quad \theta(\beta_{s_{i}})=\beta_{s_{i}};
\]
\[
\beta_{s_{i-1},s_{i}}=-\mu \prod_{\textit{T},s_{0}}(s_{i})\quad (1\leqslant i \leqslant p)\quad \text{et} \quad \theta(\beta_{s_{i-1},s_{i}})=\beta_{s_{i-1},s_{i}}.
\]
\begin{proof}
Nous procédons par récurrence sur $p$. Si $p=1$, nous appliquons le résultat de l'étape 2: $\alpha_{s_{0}s_{1}}\beta_{s_{0}}=-2\beta_{s_{0}s_{1}}$ et nous posons $\beta_{s_{0}s_{1}}=-\mu \alpha_{s_{0}s_{1}}$ avec $\mu \in K$ et nous supposons dans la suite que $\mu \neq 0$ (i. e. que la forme $\varphi$ est non nulle). On a alors $\beta_{s_{0}}=2\mu$, $\beta_{s_{1}}=2\mu \alpha_{s_{0}s_{1}}$. Ce sont les formules pour $i=0$ et $i=1$ pour $\beta_{s_{i}}$.\\
Supposons $i\geqslant 2$ et le résultat vrai pour $i$:
\[
\beta_{s_{i-1}s_{i}}=-\mu\prod_{\textit{T},s_{0}}(s_{i}),\quad \beta_{s_{i}}=2\mu\prod_{\textit{T},s_{0}}(s_{i})
\]

D'après l'étape 2, on a:
\[
\alpha_{s_{i}s_{i+1}}\beta_{s_{i}}=-2\beta_{s_{i}s_{i+1}}=2\mu(\prod_{\textit{T},s_{0}}(s_{i}))\alpha_{s_{i}s_{i+1}}
\]
donc $\beta _{s_{i}s_{i+1}}=-\mu \prod_{\textit{T},s_{0}}(s_{i+1})$ et $\beta_{s_{i+1}}=2\mu\prod_{\textit{T},s_{0}}(s_{i+1})$. Le résultat est donc vrai pour tout $i$.
\end{proof}
\textbf{Etape 4}. La dimension de l'espace vectoriel $\Phi$ est $\leqslant 1$.
\begin{proof}
Soit $t\in S$ tel que $(s_{0},t)\in E(T)$ et $s_{1}\neq t$. D'après l'étape 2, $\alpha_{s_{0}t}\beta_{s_{0}}=-2\beta_{s_{0}t}$ donc $\beta_{s_{0}t}=-\mu\alpha_{s_{0}t}$.\\
Il en résulte d'abord que $\beta_{s_{0}t}$ est non nul, puis que l'on peut appliquer l'étape 3 à toutes les branches de $\textit{T}$ qui contiennent $t$. En faisant de même pour toutes les autres branches de $\textit{T}$ qui partent de $s_{0}$, on a le résultat: $\Phi$ ne dépend que de $\mu$, donc $dim\,\Phi \leqslant 1$.
\end{proof}
\textbf{Etape 5}. La condition (b) est satisfaite.
\begin{proof}
Soient $e:=(s,t)\in E'(T)$ et $C(e)$ l'unique circuit de $\Gamma(W)$ obtenu en ajoutant l'arête $e$ à $\textit{T}$. On a 
\[
\beta_{s}=2\prod_{\textit{T},s_{0}}(s),\quad \beta_{t}=2\prod_{\textit{T},s_{0}}(t),\quad s(a_{t})=l_{st}a_{s}+a_{t},\quad t(a_{s})=a_{s}+l_{ts}a_{t}.
\]
On obtient:
\[
\beta_{st}=\varphi(a_{s},a_{t})=\varphi(s(a_{s}),s(a_{t}))=\varphi(-a_{s},l_{st}a_{s}+a_{t})=-\theta(l_{st})\beta_{s}-\beta_{st}
\]
\[
\beta_{st}=\varphi(t(a_{s}),t(a_{t}))=\varphi(a_{s}+l_{ts}a_{t},-a_{t})=-\beta_{st}-l_{ts}\beta_{t}
\]
d'où: 
\begin{equation}
\beta_{st}=-\theta(l_{st})\prod_{\textit{T},s_{0}}(s)=-l_{ts}\prod_{\textit{T},s_{0}}(t)
\end{equation}
Par un calcul semblable, on obtient:
\begin{equation}
\beta_{ts}=-l_{st}\prod_{\textit{T},s_{0}}(s)=-\theta(l_{ts})\prod_{\textit{T},s_{0}}(t)
\end{equation}
On a $\theta(\prod_{\textit{T},s_{0}}(s))=\prod_{\textit{T},s_{0}}(t)$ et $\theta(\prod_{\textit{T},s_{0}}(t))=\prod_{\textit{T},s_{0}}(s)$ d'après l'étape 3.\\
Appliquons $\theta$ à l'équation (1):
\[
\theta(\beta_{st})=-\theta^{2}(l_{st})\prod_{\textit{T},s_{0}}(s)=-\theta(l_{ts})\prod_{\textit{T},s_{0}}(t)=\beta_{ts}=-l_{st}\prod_{\textit{T},s_{0}}(s).
\]
On en déduit que $\theta(\beta_{st})=\beta_{ts}$ et $\theta^{2}(l_{st})=l_{st}$.\\
De même en appliquant $\theta$ à l'équation (2), on obtient:
 $\theta(\beta_{ts})=\beta_{st}$\\ et $\theta^{2}(l_{ts})=l_{ts}$.\\
 On a, en multipliant l'équation (1) par $l_{st}$:
\[
l_{st}\theta(l_{st})\prod_{\textit{T},s_{0}}(s)=\alpha_{st}\prod_{\textit{T},s_{0}}(t)
\]
 mais $\theta(l_{st}\theta(l_{st}))=\theta(l_{st})\theta^{2}(l_{st})=\theta(l_{st})l_{st}$, donc $\theta(l_{st})l_{st}$ est dans le sous-corps des points fixes de $\theta$. Il en est donc de même de $\alpha_{e}$. Si nous appliquons ceci à toutes les arêtes de $E'(T)$, nous voyons que $K_{0}$ est contenu dans le sous-corps des points fixes de $\theta$ et, comme $K=K_{0}(l_{e}|e\in E'(T))$, on a $\theta^{2}=id_{K}$.
\end{proof}
\textbf{Etape 6}. Fin de la démonstration.\\
Si la condition (b) est satisfaite, alors les formules obtenues précédemment montrent que $\dim \Phi =1$ et que tous les éléments de $\Phi$ sont des formes $\theta$-hermitiennes.
 \end{proof}
 \begin{remark}\label{R1}
Si $s_{p}$ est l'entrée dans le circuit $C(e)$ de l'étape 5, on obtient après simplifications, où $C(e)=(s=s_{1},\cdots,s_{p},\cdots,s_{q}=s_{t})$:
\[
\theta(l_{st})\prod_{k=1}^{p-1}\alpha_{s_{k},s_{k+1}}=l_{ts}\prod_{k=p}^{q-1}\alpha_{s_{k},s_{k+1}};
\]
si $p=1$, $\theta(l_{st})=l_{ts}\prod_{k=1}^{q-1}\alpha_{s_{k},s_{k+1}}$; si $p=q$, $\theta(l_{st})\prod_{k=1}^{q-1}\alpha_{s_{k},s_{k+1}}=l_{ts}$.
\end{remark}
\begin{corollary}\label{C1}
Si $\Gamma(W)$ est un arbre, alors $\theta = id_{K}$ (et $K=K_{0}$), $dim\,\Phi=1$ et tout $\varphi \in \Phi$ est une forme bilinéaire symétrique.
\end{corollary}
\begin{proof}
C'est une conséquence immédiate de la démonstration précédente.
\end{proof}
\begin{corollary}
Si l'on considère la représentation obtenue par la construction fondamentale, équivalente à la représentation géométrique, alors:
\begin{enumerate}
  \item $dim\, \Phi=1$ et $\theta =id_{k}$;
  \item $\varphi \in \Phi$, $\varphi$ est une forme bilinéaire symétrique.
  \end{enumerate}
\end{corollary}
\begin{proof}
Il suffit de comparer les formules.
\end{proof}
\begin{remark}
En général on a $dim\, \Phi=0$. On utilise la remarque \ref{R1}: si $\prod_{k=1}^{p-1}\alpha_{s_{k},s_{k+1}}\\\neq\prod_{k=p}^{q-1}\alpha_{s_{k},s_{k+1}}$, on prend $l_{st}=1$ et si $\prod_{k=1}^{p-1}\alpha_{s_{k},s_{k+1}}\neq\prod_{k=p}^{q-1}\alpha_{s_{k},s_{k+1}}$, on prend pour $l_{st}$ un élément tel que $\theta(l_{st})=-l_{st}$. Dans les deux cas on a une impossibilité si $\varphi \neq 0$.
\end{remark}

\textbf{Ceci explique pourquoi nous n'utiliseront pas en général les éléments de $\Phi$.}
\subsection{La matrice de Cartan}
Soit $(W,S)$ un système de Coxeter qui satisfait aux hypothèses H(Cox). On choisit un arbre couvrant de $\Gamma(W)$ ainsi qu'un élément $s_{1}$ de $S$. Pour tout $e$ arête de $\Gamma(W)$, on choisit $\alpha_{e}$ une racine de $v_{m_{e}}(X)$; pour tout $e$ arête de $E'(T)$, on choisit $l_{e}$ et $l_{e'}$ dans $K$ tels que $l_{e}l_{e'}=\alpha_{e}$. Grace à la construction fondamentale, on obtient une représentation de réflexion $R:\, W\to GL(M)$ et l'on pose $G:=Im\,R$.

La \textbf{matrice de Cartan} $Car(G)=(c_{rs})_{(r,s)\in S\times S}$ où, si $s_{j}(a_{i})=a_{i}-c_{ji}a_{j}$:
\begin{enumerate}
  \item $\forall r\in S, c_{rr}=2$;
  \item $\forall (r,s)\in S\times S, \text{si} \,rs=sr\neq 1,c(r,s)=0$;
 \item $\forall e =(r,s)\in E(T), \text{si} \,r\preccurlyeq s \,\text{alors}\, c_{rs}=-\alpha_{e} \,\text{et}\, c_{sr}=-1$;
  \item $\forall e =(r,s)\in E'(T), c_{rs}=-l_{e}\, \text{et}\, c_{sr}=-l_{e'}$.
\end{enumerate}
Par construction si $x \in M$ et si on considère la matrice $\textit{M}(x)$ dont les lignes sont les coefficients de $s_{i}(x)$, on a: 
\[
\textit{M}(x) =I_{n}x-Car(G)x.
\]
On pose: $\Delta(G):=\det Car(G) (=\Delta$ s' il ni y a pas de confusion) et on appelle $\Delta(G)$ le \textbf{discriminant} de la représentation $R$.

On a la propriété suivantes de $Car(G)$
\begin{proposition}
On garde les hypothèses et notations précédentes. Supposons qu'il existe une forme bilinéaire non nulle $G$-invariante $\varphi:M\times M \to K$. alors on a la formule:
\[
(\varphi(a_{i},a_{j}))_{1\leqslant i,j\leqslant n}= diag(1,\gamma_{2},\cdots,\gamma_{n})Car(G)
\]
où l'on a posé $\varphi(a_{i},a_{i})=2\gamma_{i}$ $(1\leqslant i \leqslant n)$.
\begin{proof}
Comme dans la partie 2 du théorème, on pose $\beta_{ij}:=\varphi(a_{i},a_{j})$ $(1\leqslant i,j \leqslant n)$ et $\beta_{i}:=\varphi(a_{i},a_{i})=2\gamma_{i}$ $(1\leqslant i\leqslant n)$. On a  d'après la démonstration du théorème 3:
\begin{itemize}
  \item si $s_{i}s_{j}=s_{j}s_{i}\neq 1$, $\beta_{ij}=0$
  \item si $e:=(s_{i},s_{j})\in E(T)$ et si $s_{i}\preccurlyeq s_{j}$, on a $\beta_{ij}=-\alpha_{e}$ et $\beta_{i}=2\gamma_{i}$
  \item si $e:=(s_{i},s_{j})\in E'(T)$, on a $\beta_{ij}=-l_{e}\gamma_{i}$.
\end{itemize}
Donc la ligne correspondant à $a_{i}$ est la même que la i-ième ligne de $Car(G)$ multipliée par $\gamma_{i}$, d'où le résultat.
\end{proof}
\end{proposition}

\subsection{La représentation duale ou contragrédiente}
Nous étudions maintenant la représentation duale $R^{*}$ de la représentation $R$.

Soit $\mathcal{A}^{*}$ = $(a_{s}^{*}|s \in S)$ la base duale de la base $\mathcal{A}$ de $M$.\\ On a $a_{s}^{*}(a_{t})=\delta_{st}\,(\forall(s,t)\in S^{2})$.\\
Pour $1\leqslant i,\,j\leqslant n$ si $S=\{s_{1},\cdots s_{n}\}$, on peut écrire $s_{i}(a_{j})=a_{j}-\lambda_{ij}a_{i}$ avec $\lambda_{ii}=2$ et si $s_{i}$ et $s_{j}$ commutent et sont distincts, $\lambda_{ij}=\lambda_{ji}=0$. Dans la base $\mathcal{A}^{*}$ on a $s_{i}(a_{j}^{*})=a_{j}^{*}$ si $i\neq j$ et $s_{i}(a_{i}^{*})=\sum_{k=1}^{n}\lambda_{ik}a_{k}^{*}$ $(1\leqslant i \leqslant n)$. On obtient alors, en posant $A'_{i}:=2a_{i}^{*}-\sum_{k=1,k\neq i}^{n}\lambda_{ik}a_{k}^{*}$ $(1\leqslant i \leqslant n)$, $\left[M^{*},s_{i}\right]=<A'_{i}>$, ce qui montre que les $s_{i}$ opèrent comme des réflexions sur $M^{*}$. Le déterminant du système de vecteurs $(A'_{i})\,(1\leqslant i \leqslant n)$ de $M^{*}$ est $\Delta(G)$. On voit donc que $(A'_{1},\cdots,A'_{n})$ est une base de $M^{*}$ si et seulement si $\Delta(G)\neq 0$, c'est à dire si et seulement si $R$ est irréductible. Nous montrons que dans ce cas la représentation $R^{*}$ est une représentation de réflexion de $W$ en trouvant une base adaptée.\\
Soient $\textit{T}$ un arbre couvrant de $\Gamma(W)$ et $s_{1}$ la racine de cet arbre. On pose $A_{1}:=A'_{1}$. On applique maintenant la construction fondamentale à la base $\textit{A'}:=(A'_{i}\,|\, 1\leqslant i \leqslant n)$ de $M^{*}$.\\
Soient $t\in S$ et $C:=(s_{1},s_{2},\cdots,s_{m}=t)$ l'unique chemin de $\textit{T}$ reliant $s_{1}$ à $t$. On pose $A_{i}:=s_{i}.A_{i-1}-A_{i-1}$ pour $2 \leqslant i \leqslant m$. Chaque $A_{i}$ est un multiple de $A'_{i}$ et on cherche le coefficient de proportionnalité.

Pour tout $i,j\, (i\neq j)$, on a $s_{i}(a_{j}^{*})=a_{j}^{*}$. On montre, par récurrence sur $j$, que $A_{j}=(\prod_{\textit{T},s_{1}}(s_{j}))A'_{j}$. Pour $j=1$, on a bien $A_{1}=A'_{1}$ par définition. Supposons le résultat vrai pour $j$. Alors $A_{j+1}=s_{j+1}.A_{j}-A_{j}=(\prod_{\textit{T},s_{1}}(s_{j}))(s_{j+1}.A'_{j}-A'_{j})$.\\
On a
\begin{multline*}
s_{j+1}.A'_{j}-A'_{j}=(-2a_{j}^{*}+\sum_{k=1,k\neq j,j+1}^{n}\lambda_{jk}a_{k}^{*}+\lambda_{j,j+1}s_{j+1}.a_{j+1}^{*})\\
+(2a_{j}^{*}-\sum_{k=1,k\neq j,j+1}^{n}\lambda_{jk}a_{k}^{*}-\lambda_{j,j+1}a_{j,j+1}^{*})
\end{multline*}
donc $s_{j+1}.A'_{j}-A'_{j}=\lambda_{j,j+1}(s_{j+1}.a_{j+1}^{*}-a_{j+1}^{*})=\lambda_{j,j+1}A'_{j+1}$ et nous avons le résultat: 
$A_{j+1}=(\prod_{\textit{T},s_{1}}(s_{j+1}))A'_{j+1}$ car $\lambda_{j,j+1}=\alpha_{j}$.\\
Soit maintenant $(s,t)\in E'(T)$. On a:
\[
A'_{s}=-2a_{s}^{*}+\sum_{u\in S,u\neq s}\lambda_{u}a_{u}^{*}, \quad A'_{t}=-2a_{t}^{*}+\sum_{v\in S,v\neq t}\mu_{v}a_{v}^{*}
\]
et, d'après le résultat précédent
\[
A_{s}=(\prod_{\textit{T},s_{1}}(s))A'_{s}, \quad A_{t}=(\prod_{\textit{T},s_{1}}(t))A'_{t}.
\]
De plus
\[
t.A'_{s}=-2a_{s}^{*}+\sum_{u\in S,u\neq s,t}\lambda_{u}a_{u}^{*}+\lambda_{t}s_{t}.a_{t}^{*}
=A'_{s}+\lambda_{t}(s_{t}.a_{t}^{*}-a_{t}^{*})=A'_{s}+l_{ts}A'_{t}
\]
donc 
\[
t.A_{s}=A_{s}+\frac {(\prod_{\textit{T},s_{1}}(s))l_{ts}}{(\prod_{\textit{T},s_{1}}(t))}A_{t}, \quad s.A_{t}=A_{t}+\frac {(\prod_{\textit{T},s_{1}}(t))l_{st}}{(\prod_{\textit{T},s_{1}}(s))}A_{s}.
\]
\section{Effet du changement de racines pour les représentations de certains groupes de Coxeter finis.}
	\subsubsection {Nous montrons comment changer la racine de l'arbre $\textit{T}$ permet de distinguer les systèmes de racines de type $B_{l}$ et $C_{l}$.}
	Nous supposons que $l=3$ car la démonstration est la même si $l\geqslant 4$ (mais plus longue à écrire!).

Nous avons le système de Coxeter (W,S) avec $S=\{s_{1},s_{2},s_{3}\}$ et le diagramme:
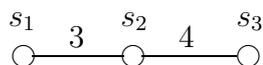
\[
\begin{picture}(150,88)
\put(29,42){\circle{7}}
\put(32,42){\line(1,0){30}}
\put(65,42){\circle{7}}
\put(68,42){\line(4,0){31}}
\put(103,42){\circle{7}}
\put(24,52){$s_{1}$}
\put(61,52){$s_{2}$}
\put(99,52){$s_{3}$}
\put(44,45){$3$}
\put(80,45){$4$}
\end{picture}
\]
qui est un arbre.\\
Nous choisissons $s_{2}$ comme racine. Nous obtenons alors la matrice de Cartan:
\[
Car(W)=
\begin{pmatrix}
2 & -1 & 0\\
-1 & 2 & -2\\
0 & -1 & 2
\end{pmatrix}
\]
et si $\varphi \in \Phi$, $\varphi$ normalisée de telle sorte que $\varphi(a_{1},a_{2})=-1$, on a
\[
(\varphi(a_{i},a_{j}))_{1\leqslant i,j\leqslant3}=
\begin{pmatrix}
2 & -1 & 0\\
-1 & 2 & -2\\
0 & -2 & 4
\end{pmatrix}
\]
Donc le diagramme de Dynkin est:
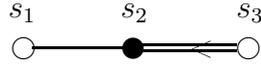
\[
\begin{picture}(150,68)
\put(29,32){\circle{7}}
\put(32,32){\line(1,0){30}}
\put(65,32){\circle*{7}}
\put(68,33){\line(1,0){31}}
\put(68,31){\line(1,0){31}}
\put(80,29){ <}
\put(103,32){\circle{7}}
\put(24,42){$s_{1}$}
\put(61,42){$s_{2}$}
\put(99,42){$s_{3}$}
\end{picture}
\]
et nous obtenons un système de racines de type $C_{3}$.\\
	2) Nous choisissons $s_{3}$ comme racine. Nous obtenons alors la matrice de Cartan:
\[
Car(W)=
\begin{pmatrix}
2 & -1 & 0\\
-1 & 2 & -1\\
0 & -2 & 2
\end{pmatrix}
\]
et si $\varphi' \in \Phi$, $\varphi'$ normalisée de telle sorte que $\varphi'(a_{1},a_{2})=-1$, on a
\[
(\varphi'(a_{i},a_{j}))_{1\leqslant i,j\leqslant3}=
\begin{pmatrix}
2 & -1 & 0\\
-1 & 2 & -1\\
0 & -1 & 1
\end{pmatrix}
\]
Donc le diagramme de Dynkin est:
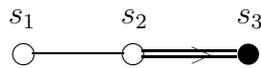
\[
\begin{picture}(150,68)
\put(29,32){\circle{7}}
\put(32,32){\line(1,0){30}}
\put(65,32){\circle{7}}
\put(68,33){\line(1,0){31}}
\put(68,31){\line(1,0){31}}
\put(80,29){ >}
\put(103,32){\circle*{7}}
\put(24,42){$s_{1}$}
\put(61,42){$s_{2}$}
\put(99,42){$s_{3}$}
\end{picture}
\]
et nous obtenons un système de racines de type $B_{3}$. \\
Soient $R$ la représentation de $W$ lorsque nous prenons $s_{2}$ comme racine et $R'$ la représentation de $W$ lorsque nous prenons $s_{3}$ comme racine. Avec les notations de la proposition 8, nous avons: $\Gamma_{1}=\{s_{1},s_{2}\}$ et $\Gamma_{2}=\{s_{3}\}$, $\alpha=\alpha_{m_{s_{2}s_{3}}}=2$, donc $g\in GL(M)$ défini par $g(a_{1})=2a_{1}$, $g(a_{2})=2a_{2}$, $ g(a_{3})=a_{3}$ est un opérateur d'entrelacement entre $R$ et $R'$. Mais $g$ est à coefficients entiers et $g^{-1}$ ne l'est pas; les représentations $R$ et $R'$ sont $\mathbb{Q}$-équivalentes mais ne sont pas $\mathbb{Z}$-équivalentes. Les formes bilinéaires symétriques $\varphi$ et $\varphi'$ ne sont pas $\mathbb{Z}$-équivalentes.

	\subsubsection{ Le groupe de Coxeter $BC_{3}$.} 
Nous faisons le même travail avec le groupe de Coxeter $BC_{3}$. Son graphe est:
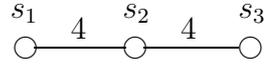
\[
\begin{picture}(150,68)
\put(29,32){\circle{7}}
\put(32,32){\line(1,0){30}}
\put(65,32){\circle{7}}
\put(68,32){\line(4,0){31}}
\put(103,32){\circle{7}}
\put(24,42){$s_{1}$}
\put(61,42){$s_{2}$}
\put(99,42){$s_{3}$}
\put(44,35){$4$}
\put(80,35){$4$}
\end{picture}
\]
	1) Si nous choisissons $s_{1}$ comme racine, on obtient la représentation $R_{1}$ et nous avons
\[
Car(G)=
\begin{pmatrix}
2 & -2 & 0\\
-1 & 2 & -2\\
0 & -1 & 2
\end{pmatrix},
\qquad \Delta(G)=0
\]
et $\Phi$ est de dimension 1 engendré par $\varphi$ dont la matrice dans la base $\emph{A}=(a_{1},a_{2},a_{3})$ est:
\[
(\varphi(a_{i},a_{j}))_{1\leqslant i,j\leqslant3}=
\begin{pmatrix}
1 & -1 & 0\\
-1 & 2 & -2\\
0 & -2 & 4
\end{pmatrix}
\]
on a donc le diagramme de Dynkin:
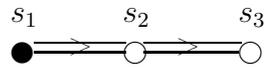
\[
\begin{picture}(150,68)
\put(28,32){\circle*{7}}
\put(32,32){\line(1,0){30}}
\put(32,35){\line(1,0){30}}
\put(65,32){\circle{7}}
\put(68,35){\line(1,0){32}}
\put(68,32){\line(1,0){31}}
\put(75,31){ >}
\put(103,32){\circle{7}}
\put(24,42){$s_{1}$}
\put(61,42){$s_{2}$}
\put(99,42){$s_{3}$}
\put(43,31){>}
\end{picture}
\]
	2) Si nous choisissons $s_{2}$ comme racine, on obtient la représentation $R_{2}$ et nous avons
\[
Car(G)=
\begin{pmatrix}
2 & -1 & 0\\
-2 & 2 & -2\\
0 & -1 & 2
\end{pmatrix},
\qquad \Delta(G)=0
\]
et $\Phi$ est de dimension 1 engendré par $\varphi$ dont la matrice dans la base $\emph{A}=(a_{1},a_{2},a_{3})$ est:
\[
(\varphi(a_{i},a_{j}))_{1\leqslant i,j\leqslant3}=
\begin{pmatrix}
2 & -1 & 0\\
-1 & 1 & -1\\
0 & -1 & 2
\end{pmatrix}
\]
on a donc le diagramme de Dynkin:
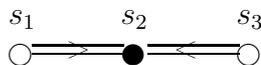
\[
\begin{picture}(150,68)
\put(28,32){\circle{7}}
\put(32,32){\line(1,0){30}}
\put(32,35){\line(1,0){30}}
\put(65,32){\circle*{7}}
\put(70,35){\line(1,0){30}}
\put(70,32){\line(1,0){30}}
\put(75,31){ <}
\put(103,32){\circle{7}}
\put(24,42){$s_{1}$}
\put(61,42){$s_{2}$}
\put(99,42){$s_{3}$}
\put(43,31){>}
\end{picture}
\]
L'opérateur d'entrelacement $g$ entre ces deux représentations est donné par:\\$g(a_{1})=2a_{1}$, $g(a_{2})=2a_{2}$, $g(a_{3})=a_{3}$. On a $K=K_{0}=\mathbb{Q}$, $R_{1}$ et $R_{2}$ sont $\mathbb{Q}$-équivalentes mais ne sont pas $\mathbb{Z}$-équivalentes.\\
	3) On ne peut pas obtenir par ce procédé la représentation $R_{3}$ qui donne le diagramme de Dynkin suivant:
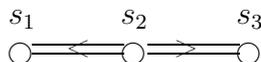
\[
\begin{picture}(150,68)
\put(28,32){\circle{7}}
\put(32,32){\line(1,0){30}}
\put(32,35){\line(1,0){30}}
\put(65,32){\circle{7}}
\put(70,35){\line(1,0){30}}
\put(70,32){\line(1,0){30}}
\put(75,31){ >}
\put(103,32){\circle{7}}
\put(24,42){$s_{1}$}
\put(61,42){$s_{2}$}
\put(99,42){$s_{3}$}
\put(43,31){<}
\end{picture}
\]
\subsubsection{Le groupe de Coxeter $W(H_{3})$.}
Nous faisons le même travail avec le groupe de Coxeter $W(H_{3})$. Son graphe est:
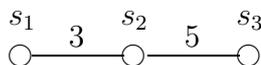
\[
\begin{picture}(150,68)
\put(28,32){\circle{7}}
\put(32,32){\line(1,0){30}}
\put(65,32){\circle{7}}
\put(70,32){\line(4,0){30}}
\put(103,32){\circle{7}}
\put(24,42){$s_{1}$}
\put(61,42){$s_{2}$}
\put(99,42){$s_{3}$}
\put(44,35){$3$}
\put(82,35){$5$}
\end{picture}
\]
Le polynôme $u_{5}(X)=v_{5}(X)=X^{2}-3X+1$ a comme racines: $\frac{3+\sqrt{5}}{2},\frac{3-\sqrt{5}}{2}$.\\
	1) Nous choisissons $s_{2}$ comme racine et $\alpha$ une racine de $u_{5}(X)$. Nous obtenons la matrice de Cartan:
\[
Car(G)=
\begin{pmatrix}
2 & -1 & 0\\
-1 & 2 & -\alpha\\
0 & -1 & 2
\end{pmatrix},
\qquad \Delta(G)=6-2\alpha
\]
et $\Phi$ est de dimension 1 engendré par $\varphi$ dont la matrice dans la base $\emph{A}=(a_{1},a_{2},a_{3})$ est:
\[
(\varphi(a_{i},a_{j}))_{1\leqslant i,j\leqslant3}=
\begin{pmatrix}
2 & -1 & 0\\
-1 & 2 & -\alpha\\
0 & -\alpha & 2\alpha
\end{pmatrix}
\]
Si on choisit $\alpha=\frac{3+\sqrt{5}}{2}$, alors on obtient le diagramme de Dynkin:
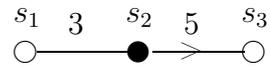
\[
\begin{picture}(150,68)
\put(28,32){\circle{7}}
\put(32,32){\line(1,0){30}}
\put(65,32){\circle*{7}}
\put(70,32){\line(1,0){30}}
\put(75,29.8){ >}
\put(103,32){\circle{7}}
\put(24,42){$s_{1}$}
\put(61,42){$s_{2}$}
\put(99,42){$s_{3}$}
\put(42,38){$3$}
\put(80,38){5}
\end{picture}
\]
Si on choisit $\alpha=\frac{3-\sqrt{5}}{2}$, alors on obtient le diagramme de Dynkin:
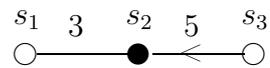
\[
\begin{picture}(150,68)
\put(28,32){\circle{7}}
\put(32,32){\line(1,0){30}}
\put(65,32){\circle*{7}}
\put(70,32){\line(1,0){30}}
\put(75,29.8){ <}
\put(103,32){\circle{7}}
\put(24,42){$s_{1}$}
\put(61,42){$s_{2}$}
\put(99,42){$s_{3}$}
\put(42,38){$3$}
\put(80,38){5}
\end{picture}
\]
	2) Nous choisissons $s_{3}$ comme racine et $\alpha$ une racine de $u_{5}(X)$. Nous obtenons la matrice de Cartan:
\[
Car(G)=
\begin{pmatrix}
2 & -1 & 0\\
-1 & 2 & -1\\
0 & -\alpha & 2
\end{pmatrix},
\qquad \Delta(G)=6-2\alpha
\]
et $\Phi$ est de dimension 1 engendré par $\varphi$ dont la matrice dans la base $\emph{A}=(a_{1},a_{2},a_{3})$ est:
\[
(\varphi(a_{i},a_{j}))_{1\leqslant i,j\leqslant3}=
\begin{pmatrix}
2 & -1 & 0\\
-1 & 2 & -1\\
0 & -1 & \alpha +1
\end{pmatrix}
\]
Si on choisit $\alpha=\frac{3+\sqrt{5}}{2}$, alors on obtient le diagramme de Dynkin:
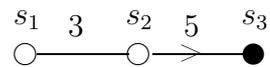
\[
\begin{picture}(150,68)
\put(28,32){\circle{7}}
\put(32,32){\line(1,0){30}}
\put(65,32){\circle{7}}
\put(70,32){\line(1,0){30}}
\put(75,29.8){ >}
\put(103,32){\circle*{7}}
\put(24,42){$s_{1}$}
\put(61,42){$s_{2}$}
\put(99,42){$s_{3}$}
\put(42,38){$3$}
\put(80,38){5}
\end{picture}
\]
Si on choisit $\alpha=\frac{3-\sqrt{5}}{2}$, alors on obtient le diagramme de Dynkin:
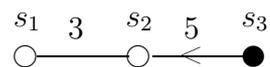
\[
\begin{picture}(150,68)
\put(28,32){\circle{7}}
\put(32,32){\line(1,0){30}}
\put(65,32){\circle{7}}
\put(70,32){\line(1,0){30}}
\put(75,29.8){ <}
\put(103,32){\circle*{7}}
\put(24,42){$s_{1}$}
\put(61,42){$s_{2}$}
\put(99,42){$s_{3}$}
\put(42,38){$3$}
\put(80,38){5}
\end{picture}
\]Nous voyons ainsi que changer de racine revient à changer $\alpha$ en $3-\alpha$, c'est à dire que l'on obtient la représentation conjuguée de la représentation $R$.

\begin{center}
Université de Picardie Jules Verne\\
 Pôle Scientifique\\
Laboratoire LAMFA, UMR CNRS 7352\\
33, rue Saint Leu\\
80039 Amiens Cedex\\
francois.zara@u-picardie.fr
\end{center}


\begin{thebibliography}{9}
\bibitem{B}
N. Bourbaki
\emph{Groupes et algèbres de Lie. Chapitres 4, 5, 6}
Hermann (1968)
\bibitem{BMR}
M. Broué, G. Malle, R. Rouquier
\emph{Complex reflection groups, braids groups, Hecke algebra}
J. Reine Angew. Math. 500 (1988) 127--190
\bibitem{BR}
M. Broué
\emph{Introduction to Complex Reflection Groups and their Braids Groups}
Lecture notes in Mathematics 1988 Springer Verlag (2010)
\bibitem{Co}
A. M. Cohen
  \emph{Finite Complex Reflection Groups}
  Ann. scient. Ec. Norm. Sup. t. 9 (1976) 379--436
 \bibitem{GP}
 M. Geck, G. Pfeiffer
 \emph{Characters of Finite Coxeter Groups and Iwahori-Hecke Algebras}
 London Math. Soc. monographs new Series 21
 Clarendon Press Oxford (2000)
 \bibitem{K}
 R. Kane
 \emph{Reflection Groups and Invariant Theory}
 Société Mathématique du Canada
 Springer Verlag (2001)
 \bibitem{ LT}
 G. Lehrer, D. Taylor
 \emph{Unitary Reflection Groups}  
 Australian Mathematical Societe Series 20
 Cambridge (2009)
  \bibitem{O}
  Oystien Ore
  \emph{Theory of Graphs}
  Amer. Math. Soc. (1962) Colloquium publications 38
    \bibitem{Z}
  F. Zara
  \emph{Generalized reflection groups}
  Journal of Algebra 255 (2002) 221--246
\end{thebibliography}
\end{document}